\documentclass{amsart}
\usepackage{verbatim}
\usepackage{amssymb}
\usepackage{amsmath}
\usepackage[usenames]{color}
\usepackage{graphicx}
\usepackage{amscd}
\usepackage[numbers,sort&compress,square]{natbib}

\theoremstyle{plain}
\newtheorem{theorem}{Theorem}

\newtheorem{proposition}[theorem]{Proposition}
\newtheorem{lemma}[theorem]{Lemma}
\newtheorem{corollary}[theorem]{Corollary}

\newtheorem{definition}[theorem]{Definition}

\theoremstyle{definition}

\newtheorem{remark}[theorem]{Remark}

\newtheorem{example}[theorem]{Example}
\numberwithin{equation}{section}
\numberwithin{theorem}{section}
\allowdisplaybreaks

\newcommand{\eps}{\varepsilon}

\newcommand{\dd}[0]{\mathrm{d}}
\newcommand{\ud}[0]{\,\mathrm{d}}

\newcommand{\sign}{\text{sign\,}}

\begin{document}

\title[Even Fourier multipliers and martingale transforms]
{Even Fourier multipliers and martingale transforms in infinite dimensions}

\author{Ivan S. Yaroslavtsev}
\address{Delft Institute of Applied Mathematics\\
Delft University of Technology \\ P.O. Box 5031\\ 2600 GA Delft\\The
Netherlands}
\email{I.S.Yaroslavtsev@tudelft.nl}

\begin{abstract}
In this paper we show sharp lower bounds for norms of even homogeneous Fourier multipliers in $\mathcal L(L^p(\mathbb R^d; X))$  for $1<p<\infty$ and for a~UMD Banach space $X$ in terms of the range of the corresponding symbol. For example, if the range contains $a_1,\ldots,a_N \in \mathbb C$, then the norm of the multiplier exceeds $\|a_1R_1^2 + \cdots + a_NR_N^2\|_{\mathcal L(L^p(\mathbb R^N; X))}$, where $R_n$ is the corresponding Riesz transform. We also provide sharp upper bounds of norms of Ba\~{n}uelos-Bogdan type multipliers in terms of the range of the functions involved. The main tools that we exploit are $A$-weak differential subordination of martingales and UMD$_p^A$ constants, which are introduced here.
\end{abstract}


\maketitle

\section{Introduction}

Martingale transform norms play a significant r\^ole in Fourier multiplier theory e.g.\ to obtain sharp bounds for different kinds of Fourier multipliers. The most natural is the norm of a transform $T_{\{-1,1\}}$, which is defined as follows: for a discrete martingale $f=(f_n)_{n\geq 0}$ we set
\begin{equation}\label{eq:defofT-11}
  \bigl(T_{\{-1,1\}}f\bigr)_n:= f_0 + \sum_{k=1}^{n}(-1)^k (f_k-f_{k-1}).
\end{equation}
Due to the fundamental work \cite{Burk84} of Burkholder, the norm of this operator acting on scalar-valued $L^p$-integrable martingales is known to be equal to $p^*-1$ (where $p^* := \max\{p,\frac{p}{p-1}\}$). For this reason $p^*-1$ appears naturally as an upper bound of the Hilbert transform (see Burkholder \cite{Burk89}), and of Ba\~{n}uelos-Bogdan type multipliers (which are also known as {\it L\'evy multipliers}, see Ba\~{n}uelos and Bogdan \cite{BB}, and Ba\~{n}uelos, Bielaszewski, and Bogdan \cite{BBB}). Therefore it is reasonable to extend such multipliers to the so-called {\it UMD Banach spaces}, which are defined in the following way: $X$ is called a UMD Banach space if for some (equivalently, for all) $1<p<\infty$ the operator $T_{\{-1,1\}}$ defined by \eqref{eq:defofT-11} is bounded as an operator on the space of all $X$-valued $L^p$-martingales (an equivalent classical definition is given in \eqref{eq:defofUMDconstant}, the equivalence can be shown by \cite[Lemma 2.1]{Burk84}); the corresponding norm of $T_{\{-1,1\}}$ is then called the {\it UMD$^{\{-1,1\}}_p$ constant} of $X$ and is denoted by $\beta_{p,X}^{\{-1,1\}}$. 
It turns out that the UMD property is necessary and sufficient for the boundedness of the Hilbert transform on $L^p(\mathbb R; X)$ (see Bourgain \cite{Bour83} and Garling \cite{Gar85}). Moreover, the UMD property is equivalent to the boundedness of all Mihlin multipliers on $L^p(\mathbb R^d; X)$ (see McConnell \cite{McC84} and Bourgain \cite{Bour}), and that of Ba\~{n}uelos-Bogdan multipliers on $L^p(\mathbb R^d; X)$ (see \cite{Y17FourUMD}) for $d\geq 2$; furthermore, only in UMD Banach spaces can
nontrivial even homogeneous Fourier multipliers be bounded (see Geiss, Montgomery-Smith, and 
Saksman \cite{GM-SS}). In the latter two approaches the norm of the martingale transform $T_{\{-1,1\}}$ is of big importance as a~sharp upper (respectively, lower) bound. The main goal of this article is to consider the norm of a more general martingale transform $T_{A}$ depending on a bounded set $A\subset \mathbb C$ (instead of just $\{-1,1\}$), and to extend the assertions from \cite{Y17FourUMD} and \cite{GM-SS} towards assertions depending on $A$. We will show that $T_A$, which is defined analogously to $T_{\{-1,1\}}$ (see Theorem \ref{thm:Atransform}), is bounded if and only if $X$ is a~UMD Banach space, and that the corresponding norm is comparable with the UMD$^{\{-1,1\}}_p$ constant; we call this norm the {\it UMD$_p^A$ constant} of $X$ and denote it by $\beta_{p,X}^A$. The basic properties of UMD$_p^A$ constants are discussed in Proposition \ref{prop:propertiesofUMP_p,X^A}. 
One of these properties is pretty obvious: $\beta_{p, X}^{A_1}\leq \beta_{p, X}^{A_2}$ if $A_1\subset A_2$, and it is connected to the following open problem: whether $\beta_{p,X}^{\{-1,1\}} = \beta_{p, X}^A$ for any symmetric set $A$ of diameter~$2$ (e.g.~for the unit disk). This is known to be true in the scalar-valued and in the Hilbert space-valued setting (see \cite[Corollary 4.5.15]{HNVW1}). As we will see later in Remark~\ref{rem:BAtranssharpbounds}, this problem affects the lower bound of the Beurling-Ahlfors transform in the infinite dimensional case.

\smallskip

Using properties of UMD$_p^A$ constants we extend the estimate from \cite[Proposition 3.4]{GM-SS} and prove that the norm $\|T_m\|_{\mathcal L(L^p(\mathbb R^d;X))}$ of a Fourier multiplier $T_m$ with even homogeneous symbol $m:\mathbb R^d \to \mathbb C$ continuous on $\mathbb R^d\setminus \{0\}$ is bounded from below by $\beta^{\text{Ran}\left(m|_{\mathbb R^d\setminus \{0\}}\right)}_{p,X}$, where $\text{Ran}(m|_{\mathbb R^d\setminus \{0\}})$ is the range of $m$ on $\mathbb R^d\setminus \{0\}$. The sharpness of this estimate follows from Example \ref{ex:a_1a_dalphain02}, where we show in particular that
\begin{equation}\label{eq:examplelinconbsecorderriesztrans}
  \|a_1R_1^2 + \ldots + a_d R_d^2\|_{\mathcal L(L^p(\mathbb R^d;X))} = \beta_{p,X}^{\{a_1,\ldots,a_d\}},
\end{equation}
where $R_j$ is the corresponding Riesz transform.

Another interesting family of multipliers was introduced by Ba\~{n}uelos and Bogdan in \cite{BB}. This family is defined through the symbol in the following way
 \begin{equation*}
  m(\xi) = \frac{\int_{\mathbb R^d} (1-\cos (\xi \cdot z))\phi(z)V(\dd z) + 
\frac 12\int_{S^{d-1}}(\xi \cdot \theta)^{2}\psi(\theta)\mu(\dd\theta)}{\int_{\mathbb 
R^d} (1-\cos (\xi \cdot z))V(\dd z)+ \frac 12\int_{S^{d-1}}(\xi \cdot 
\theta)^{2}\mu(\dd\theta)},\;\; \xi \in \mathbb R^d,
 \end{equation*}
 where $V$ is a L\'evy measure, $\mu$ is a bounded measure on the unit sphere $S^{d-1}$, $\phi \in L^{\infty}(\mathbb R^d)$ and $\psi \in L^{\infty}(S^{d-1})$. In \cite{BB}, and later more generally in \cite{BBB} it was shown that if $\|\phi\|_{\infty}, \|\psi\|_{\infty}\leq 1$, then $\|T_m\|_{\mathcal L(L^p(\mathbb R^d))}\leq p^*-1$. In \cite{Y17FourUMD} this result was generalized to the UMD space case and it was proven that $\|T_m\|_{\mathcal L(L^p(\mathbb R^d;X))}\leq \beta_{p, X}^{\mathbb D}$ (here $\mathbb D$ is the unit disk in $\mathbb C$). In the current article we will show that if both $\phi$ and $\psi$ have values in a bounded set $A\subset \mathbb C$, then
 \begin{equation}\label{eq:introBBT_mupperbound}
  \|T_m\|_{\mathcal L(L^p(\mathbb R^d;X))}\leq \beta_{p, X}^{A},
 \end{equation}
 and the sharpness of these estimates again follows from \eqref{eq:examplelinconbsecorderriesztrans}.
 
 An important tool for proving \eqref{eq:introBBT_mupperbound} is so-called {\it $A$-weak differential subordination} of martingales. This is generalization of the notion of weak differential subordination which was introduced in \cite{Y17FourUMD}, and can be characterized in the case of discrete martingales in the following way: an $X$-valued martingale $(g_n)_{n\geq 0}$ is $A$-weakly differentially subordinated to an $X$-valued martingale $(f_n)_{n\geq 0}$ if there exists an adapted sequence $(a_n)_{n\geq 0}$ with values in $A$ such that $g_0 = a_0 f_0$ and $g_n-g_{n-1} = a_n(f_n-f_{n-1})$ for each $n\geq 1$. If this is the case, then due to Theorem \ref{thm:AWDSofdiscretemartingales} below
 \[
  \mathbb E \|g_n\|^p \leq (\beta_{p, X}^A)^p \mathbb E \|f_n\|^p,\;\;\; n\geq 0,
 \]
where $\beta_{p, X}^A$ is the UMD$_p^A$ constant of $X$. An analogous statement holds for purely discontinuous martingales, see Theorem \ref{thm:bBwdsforpurdisc}.
 
 \smallskip
 
 We wish to give the reader a short historical overview on types of martingale transforms different from \eqref{eq:defofT-11}. The first progress in this direction in the scalar-valued case was done by Choi in \cite{Choi92}, where he in fact worked with the martingale transform $T_{\{0,1\}}$. The norm of such a martingale transform equals $\beta_{p,\mathbb R}^{\{0,1\}}$ (which approximately equals $\frac p2 + \frac 12 \log\bigl(\frac{1+e^{-2}}{2}\bigr)$, see \cite[Theorem 4.3]{Choi92}), and is called the {\it Choi constant}. Further, Ba\~{n}uelos and Os{\c{e}}kowski in \cite{BO12} analyzed the norm of $T_{\{b,B\}}$ in the scalar case for arbitrary $b,B \in \mathbb R$, and gave sharp lower and upper bounds of Fourier multipliers, similar to those presented in Theorem \ref{thm:multlowerbound} and \ref{thm:multupperbound}, in terms of this norm. Finally, in \cite{HNVW1} Hyt\"onen, van Neerven, Veraar, and Weis working with UMD spaces introduced the UMD constant in a different way: they used all the numbers $a\in \mathbb C$ such that $|a|=1$ instead of just $-1$ and $1$. (This definition will be used further in the paper). Even though in the scalar-valued case it does not give a significant advantage (see \cite[Corollary 4.5.15]{HNVW1}), it leads to more general assertions in infinite dimensions, in particular to assertions on transforms of type $T_{\mathbb D}$.
 
 \smallskip
 
 In the end we wish to point the reader's attention to the fact that although this paper provides extensions of well-known results from \cite{BB,BBB,GM-SS,Choi92,BO12}, the direction of  these extensions is new even in the scalar-valued case.

 \smallskip
 
\emph{Acknowledgment} -- The author would like to thank Mark Veraar for posing the question concerning generalization of UMD constants and  for his useful suggestions, in particular for showing Remark \ref{rem:conterexevenhomandunbdd}. The author thanks Tuomas Hyt\"onen for a~short fruitful discussion on the topic of the paper, Wolter Groenevelt for discussing Example~\ref{exFourmultGroenevelt}, and Alex Amenta for his helpful comments.

\section{Preliminaries}

In the sequel we set the scalar field $\mathbb K$ be either $\mathbb R$ or 
$\mathbb C$, unless stated otherwise.

\smallskip

A Banach space $X$ over $\mathbb K$ is called a {\it UMD Banach space} if for some (equivalently, for 
all)
$p \in (1,\infty)$ there exists a constant $\beta>0$ such that
for every $N \geq 1$, every martingale
difference sequence $(d_n)^N_{n=1}$ in $L^p(\Omega; X)$, and every scalar-valued 
sequence
$(\varepsilon_n)^N_{n=1}$ such that $|\varepsilon_n|=1$ for each $n=1,\ldots,N$
we have
\begin{equation}\label{eq:defofUMDconstant}
 \Bigl(\mathbb E \Bigl\| \sum^N_{n=1} \varepsilon_n d_n\Bigr\|^p\Bigr )^{\frac 
1p}
\leq \beta \Bigl(\mathbb E \Bigl \| \sum^N_{n=1}d_n\Bigr\|^p\Bigr )^{\frac 1p}.
\end{equation}
The least admissible constant $\beta$ is denoted by $\beta_{p,X}$ and is called 
the {\it UMD$_p$~constant} or, if the value of $p$ is understood, the {\em UMD constant} of $X$.
It is well-known that UMD spaces obtain a large number 
of useful properties, such as being reflexive. Examples of UMD 
spaces include all finite dimensional spaces and the reflexive range of 
$L^q$-spaces, Besov spaces, Sobolev spaces and Schatten class spaces. Example of 
spaces without the UMD property include all nonreflexive Banach spaces, e.g.\ 
$L^1(0,1)$ or $C([0,1])$. We refer the reader to 
\cite{Burk01,HNVW1,Rubio86,Pis16} for details.

\smallskip

Let $X$ be a Banach space, $\mathbb F$ be a filtration, $(f_n)_{n\geq 0}$ be an $X$-valued local martingale. For 
each $n\geq 1$ we define $df_n := f_n - f_{n-1}$, $df_0 := f_0$. Let $1\leq p<\infty$. An $X$-valued martingale $(f_n)_{n\geq 0}$ is called an {\it $L^p$-integrable martingale} (or just an {\it $L^p$-martingale}) if $f_n \in L^p(\Omega;X)$ for each $n\geq 0$, and there exists a limit $f_{\infty} = \lim_{n\to\infty} f_n$ in $L^p(\Omega; X)$. We will denote the linear space of all $X$-valued $L^p$-integrable $\mathbb F$-martingales by $\mathcal M^{\rm dis}_{p}(X,\mathbb F)$. Notice that $\mathcal M^{\rm dis}_{p}(X,\mathbb F)$ is a Banach space with the norm
\[
 \|(f_n)_{n\geq 0}\|_{\mathcal M^{\rm dis}_{p}(X,\mathbb F)} := \|f_{\infty}\|_{L^p(\Omega;X)} = \lim_{n\to \infty} \|f_{n}\|_{L^p(\Omega;X)}.
\]
For simplicity we will omit $\mathbb F$ and denote $\mathcal M^{\rm dis}_{p}(X,\mathbb F)$ by $\mathcal M^{\rm dis}_{p}(X)$. Details on $\mathcal M^{\rm dis}_{p}(X)$ can be found in \cite[Section 3.1]{HNVW1}.

\smallskip

 A random variable $r:\Omega \to \{-1, 1\}$ is called a~{\it Rademacher variable} if 
 $$
 \mathbb P(r=1) = \mathbb P(r=-1) = \frac 12.
 $$

\begin{definition}[Paley-Walsh martingale]\label{def:Paley-Walsh}
 Let $X$ be a Banach space. A discrete $X$-valued martingale $(f_n)_{n\geq 0}$ is called Paley-Walsh if $f_0$ is a constant, and if there exist a sequence of independent Rademacher variables $(r_n)_{n\geq 1}$, a function $\phi_n:\{-1, 1\}^{n-1}\to X$ for each $n\geq 2$ and $\phi_1 \in X$ such that $df_n = r_n \phi_n(r_1,\ldots, r_{n-1})$ for each $n\geq 2$ and $df_1 = r_1 \phi_1$.
\end{definition}

\begin{remark}\label{rem:UMDthrough-11trans}
Let $X$ be a Banach space over $\mathbb R$, $1<p<\infty$. Then $\beta_{p, X}$ can be represented as a norm of a certain operator acting on $\mathcal M_{p}^{\rm dis}$ (which is finite if and only if $X$ has the UMD property). Namely, if we fix a~Paley-Walsh filtration $\mathbb F  = (\mathcal F_n)_{n\geq 0}$ generated by a countable sequence of independent Rademacher random variables and fix a martingale transform $T_{\{-1,1\}}$ acting on an~$X$-valued $\mathbb F$-martingale $f=(f_n)_{n\geq 0}$ as follows:
\[
 d(T_{\{-1,1\}}f)_{n} = (-1)^n df_n,\;\;\; n\geq 0,
\]
then $\|T_{\{-1,1\}}\|_{\mathcal L(\mathcal M^{\rm dis}_p(X))} = \beta_{p,X}$ (see \cite{Burk84} and Theorem \ref{thm:Atransform} below).
\end{remark}

\section{UMD$_p^A$ constants}

\subsection{Definition and basic properties}
Let $X$ be a Banach space over the scalar field 
$\mathbb K$, $1<p<\infty$, $A \subset \mathbb K$ (recall that $\mathbb K$ is either $\mathbb R$ or $\mathbb C$). Then we define the {\it UMD$_p^A$ constant} $\beta_{p,X}^A\in [0,\infty]$ of $X$ as the least nonnegative number $\beta$ such that for each $L^p(\Omega; X)$-valued martingale difference sequence $(d_n)_{n=1}^N$ and any sequence $(\eps_n)_{n=1}^N$ with values in $A$
\begin{equation}\label{eq:UMD_p^Aconstant}
 \Bigl(\mathbb E \Bigl\|\sum_{n=1}^N \eps_n d_n\Bigr\|^p\Bigr)^{\frac 1p} \leq \beta  \Bigl(\mathbb E \Bigl\|\sum_{n=1}^N d_n\Bigr\|^p\Bigr)^{\frac 1p}.
\end{equation}

Let us outline basic properties of UMD$_p^A$ constants. Recall that
\begin{itemize}
 \item $\textnormal{Conv}(A)$ is the convex hull of a set $A$,
 \item $\overline{A}$ is a closure of a set $A$,
 \item $\textnormal{Diam}(A)$ is the diameter of $A$, i.e.\ $\sup_{a_1, a_2 \in A} |a_1-a_2|$,
 \item $\text{Conj}(A):= \{\bar{a}:a\in A\}$ is the conjugate of $A$,
 \item $A_1+A_2 = \{a\in \mathbb K: a=a_1+a_2\; \textnormal{for some} \;a_1\in A_1, a_2\in A_2\}$ is the Minkowski sum of sets $A_1$, $A_2 \subset \mathbb K$.
\end{itemize}

\begin{proposition}\label{prop:propertiesofUMP_p,X^A}
 Let $X$ be a Banach space, $1<p<\infty$, $A$, $A_1$, $A_2  \subset \mathbb K$, $b$, $B\in \mathbb R$ be such that $b<B$. Then the following holds:
 \begin{itemize}
  \item [(i)] $\beta_{p,X}^{aA}=\beta_{p,X}^{a\textnormal{Conj}(A)} = |a|\beta_{p,X}^{A}$ for each $a\in \mathbb K$. In particular, we have that $\beta_{p,X}^{-A} = \beta_{p,X}^{\textnormal{Conj}(A)} = \beta_{p,X}^{A}$.
  \item [(ii)] $\beta_{p,X}^A = \beta_{p,X}^{\overline{\textnormal{Conv}(A)}}$.
  \item [(iii)] If $\overline{\textnormal{Conv}(A_1)} \subset \overline{\textnormal{Conv}(A_2)}$, then $\beta_{p,X}^{A_1} \leq \beta_{p,X}^{A_2}$. In particular, if $A_1 \subset A_2$, then $\beta_{p,X}^{A_1} \leq \beta_{p,X}^{A_2}$.
  \item [(iv)] If $A$ is unbounded, then $\beta_{p,X}^A = \infty$.
  \item [(v)] If $A = \{a\}$ consists of one point, then $\beta_{p,X}^A = |a|$.
  \item [(vi)] $\beta_{p,X}^{ {A_1+A_2}} \leq \beta_{p,X}^{A_1} + \beta_{p,X}^{A_2}$.
  \item [(vii)] If $A$ is bounded, then
  \begin{equation}\label{eq:estimatesforUMD_p,X^AviaUMDconstant}
   \max\Bigl\{\frac{1}{\pi} \textnormal{Diam}(A) \beta_{p,X}, \sup_{a\in A} |a|\Bigr\} \leq\beta_{p,X}^A \leq \sup_{a\in A} |a| \beta_{p, X},\;\;\; \text{if } \mathbb K = \mathbb C,
  \end{equation}
   \begin{equation}\label{eq:estimatesforUMD_p,X^b,B}
   \max\Bigl\{\frac {B-b}{2}\beta_{p,X}, |b|, |B|\Bigr\} \leq   \beta_{p,X}^{A} \leq \frac {B-b}{2}\beta_{p,X} + \frac {|B+b|}{2},\;\;\; \text{if } \mathbb K = \mathbb R,
  \end{equation}
  where $b = \inf_{a\in A}a$, $B = \sup_{a\in A}a$ in the case $\mathbb K = \mathbb R$.
In particular, if $A$ is bounded and contains at least two points, then $\beta_{p,X}^A$ is finite if and only if $\beta_{p,X}$ is finite, i.e.\ if and only if $X$ has the UMD property.
\item[(viii)] For any scalar field $\mathbb K$
 \begin{equation}\label{eq:thm:bddnessofpiimpUMD1}
\max\Bigl\{\frac {B-b}{2}\beta_{p,X}^{\{-1,1\}}, |b|, |B|\Bigr\} \leq   \beta_{p,X}^{\{b,B\}} \leq \frac {B-b}{2}\beta_{p,X}^{\{-1,1\}} + \frac {|B+b|}{2}.
 \end{equation}
 \item[(ix)]  $\beta_{p,X}^A = \beta_{p',X^*}^A$.
 \item[(x)] Let $(A_n)_{n\geq 1}$ be such that $A_n\subset \mathbb K$ for each $n\geq 1$. If $(A_n)_{n\geq 1}$ is increasing and $\cup_n A_n =A$, then $\beta_{p, X}^{A_n}\nearrow \beta_{p,X}^A$ as $n\to \infty$. If $(A_n)_{n\geq 1}$ is decreasing, the sets $(A_n)_{n\geq 1}$ are bounded, and $\cap_n A_n= A$, then $\beta_{p, X}^{A_n}\searrow \beta_{p,X}^A$ as $n\to \infty$.
 \end{itemize}
\end{proposition}

For the proof we will need the following technical lemma, proven in \cite[Proposition 4.2.10]{HNVW1}. Recall that $\mathbb D = \{a\in \mathbb C:|a|\leq 1\}$ is the unit disk in $\mathbb C$. 

\begin{lemma}\label{UMDconstant-11andDcomparison}
 Let $\mathbb K = \mathbb C$. Then
 \[
  \beta_{p,X}^{\{-1,1\}} \leq \beta_{p,X}^{\mathbb D} = \beta_{p,X}\leq \frac{\pi}{2}\beta_{p,X}^{\{-1,1\}}.
 \]
\end{lemma}

\begin{proof}[Proof of Proposition \ref{prop:propertiesofUMP_p,X^A}]
 (i), (iv) and (v) follow directly from the definition given in \eqref{eq:UMD_p^Aconstant}, while (ii) and (vi) follow from \eqref{eq:UMD_p^Aconstant} and the triangle inequality. (iii) can be derived from (ii) and \eqref{eq:UMD_p^Aconstant}. Let us show (vii)-(x).
 
 (vii): First assume that $\mathbb K = \mathbb C$. The inequality $\beta_{p,X}^A \leq \sup_{a\in A} |a| \beta_{p, X}$ follows from (i), (iii), the fact that $A \subset  \sup_{a\in A}|a|\mathbb D$, where $\mathbb D$ is the unit disk in $\mathbb C$, and the fact that $\beta_{p,X}^{\mathbb D} = \beta_{p,X}$ by \eqref{eq:defofUMDconstant}. The inequality $\beta_{p,X}^A \geq \sup_{a\in A} |a|$ follows from~\eqref{eq:UMD_p^Aconstant}. 
 
 Let us now show that $\beta_{p,X}^A \geq \frac{1}{\pi} \textnormal{Diam}(A) \beta_{p,X}$. Due to (ii) we may assume that $A$ is convex and closed. Let $a_1, a_2 \in A$ be such that $|a_1-a_2| = \textnormal{Diam}(A)$. Then by (iii) we know that $\beta_{p, X}^A \geq \beta_{p,X}^{\{a_1,a_2\}}$. Further, $\beta_{p,X}^{\{a_1,a_2\}} \geq \frac{|a_1-a_2|}{2}\beta_{p,X}^{\{-1,1\}}$. Indeed, for each $L^p(\Omega; X)$-valued martingale difference sequence $(d_n)_{n=1}^N$ and any $\{-1,1\}$-valued sequence $(\eps_n)_{n=1}^N$

 \begin{equation}\label{eq:a_1a_2lowerbound}
  \begin{split}
   \frac {|a_1-a_2|}{2}\Bigl(\mathbb E \Bigl\|\sum_{n=1}^N \eps_n d_n\Bigr\|^p\Bigr)^{\frac 1p} &= \Bigl(\mathbb E \Bigl\|\sum_{n=1}^N \frac {a_1-a_2}{2}\eps_n d_n\Bigr\|^p\Bigr)^{\frac 1p}\\
&=\Bigl(\mathbb E \Bigl\|\sum_{n=1}^N \frac 12\Bigl(\frac {a_1-a_2}{2}\eps_n + \frac {a_1+a_2}{2}\Bigr) d_n\\
&\quad+ \sum_{n=1}^N \frac 12\Bigl(\frac {a_1-a_2}{2}\eps_n - \frac {a_1+a_2}{2}\Bigr) d_n\Bigr\|^p\Bigr)^{\frac 1p}\\
&\leq\frac 12\Bigl(\mathbb E \Bigl\|\sum_{n=1}^N \Bigl(\frac {a_1+a_2}{2}+\frac {a_1-a_2}{2}\eps_n\Bigr) d_n\Bigr\|^p\Bigr)^{\frac 1p}\\
&\quad+\frac 12\Bigl(\mathbb E \Bigl\| \sum_{n=1}^N \Bigl(\frac {a_1+a_2}{2} -\frac {a_1-a_2}{2}\eps_n \Bigr) d_n\Bigr\|^p\Bigr)^{\frac 1p}\\
&\stackrel{(*)}\leq \beta_{p,X}^{\{a_1,a_2\}}\Bigl(\mathbb E \Bigl\|\sum_{n=1}^N d_n\Bigr\|^p\Bigr)^{\frac 1p},
  \end{split}
 \end{equation}
where $(*)$ follows from the definition of $\beta_{p,X}^{\{a_1,a_2\}}$ and the fact that 
$$
\frac {a_1+a_2}{2}+\frac {a_1-a_2}{2}\eps, \frac {a_1+a_2}{2} -\frac {a_1-a_2}{2}\eps \in \{a_1, a_2\}
$$
for each $\eps\in \{-1,1\}$. By Lemma \ref{UMDconstant-11andDcomparison} we have that $\beta_{p,X}^{\{-1,1\}} \geq \frac {2}{\pi}\beta_{p,X}$, therefore due to all the estimates above
\[
 \beta_{p,X}^A \geq \beta_{p,X}^{\{a_1,a_2\}} \geq\frac{\textnormal{Diam}(A)}{2}\beta_{p,X}^{\{-1,1\}} \geq \frac{1}{\pi} \textnormal{Diam}(A) \beta_{p,X}.
\]

Now turn to the proof of \eqref{eq:estimatesforUMD_p,X^b,B}. The lower bound of $\beta_{p,X}^{\{b,B\}}$ follows analogously the lower estimate of $\beta_{p,X}^{A}$ in \eqref{eq:estimatesforUMD_p,X^AviaUMDconstant}, but we now use that $\beta_{p,X}^{\{-1,1\}} = \beta_{p,X}$ since $\mathbb K = \mathbb R$. The upper bound follows analogously the upper bound in \cite[Theorem~4.1]{Choi92}.
 
 (viii) can be shown analogously \eqref{eq:estimatesforUMD_p,X^b,B}.
 
 (ix): The proof is similar to the one of \cite[Proposition 4.2.17(2)]{HNVW1}.
 
 (x): Without loss of generality by (ii) we can assume that the sets $(A_n)_{n\geq 1}$ are convex and closed. Moreover, we can assume that all the sets are bounded (otherwise the statement is evident). Let us prove the first part, the second will follow analogously. From the one side, by (iii) $\beta_{p, X}^{A_n}\leq \beta_{p,X}^A$ for each $n\geq 1$. From the other side, by compactness of $(A_n)_{n\geq 1}$ and $\overline A$ for any fixed $\delta>0$ there exists $N_{\delta}\geq 1$ such that $A\subset A_n + \delta \mathbb D$ for each $n\geq N_{\delta}$ (recall that $\mathbb D$ is the unit disk in $\mathbb C$). Then by (i) and (vi), for each $n\geq N_{\delta}$
 \[
  \beta_{p,X}^A \leq \beta_{p,X}^{A_n} + \delta \beta_{p,X}^{\mathbb D}.
 \]
Due to (iii), the sequence $(\beta_{p,X}^{A_n})_{n\geq 1}$ is monotone, so according to the arguments above it must converge to $\beta_{p,X}^A$.
\end{proof}

\begin{remark}\label{rem:beta-11vsbetaD}
 The inequality presented in Lemma \ref{UMDconstant-11andDcomparison} is not known to be sharp. Moreover, it is an open problem, which was posed to the author by Jan van Neerven and Mark Veraar, whether one has that $\beta_{p, X}^{\mathbb D} = \beta_{p,X}^{\{-1,1\}}$. The issue is not only in the proper theoretical proof or disproof of the equality $\beta_{p, X}^{\mathbb D} = \beta_{p,X}^{\{-1,1\}}$, but also in the fact that the UMD constants $\beta_{p, X}^{\mathbb D}$ and $\beta_{p,X}^{\{-1,1\}}$ are not properly computed except the case of $L^q(S)$ for $q\in [p,2]$ if $p\leq 2$, and $q\in[2,p]$ if $p\geq 2$ (then both $\beta_{p, X}^{\mathbb D}$ and $\beta_{p,X}^{\{-1,1\}}$ are known to be $p^*-1$, see \cite[Proposition 4.2.22]{HNVW1}; notice that this includes the case of $L^p(S)$ and the Hilbert space-case).

 This question is connected to a precise geometric definition of $\beta_{p, X}^{\mathbb D}$ and $\beta_{p,X}^{\{-1,1\}}$ (and $\beta_{p,X}^A$ for a general set $A\subset \mathbb K$) in terms of the unit ball of $X$; then both $\beta_{p, X}^{\mathbb D}$ and $\beta_{p,X}^{\{-1,1\}}$ would be efficiently computable on a machine with an acceptable accuracy (for a simple or well-discovered finite dimensional space $X$, e.g.\ $X = \ell^{\infty}_n$), and the problem could be solved. 
 
 Concerning this problem we wish to notice that in the case of $A = \{0,1\}$ and $X$ being a Hilbert space, $\beta_{p,X}^{A}$ has a series representation (see \cite[Theorem 4.3]{Choi92}). It remains open whether such a series representation is possible for a general space $X$ or for a general set $A$.
\end{remark}

The following theorem is similar to Remark \ref{rem:UMDthrough-11trans} and \cite[Lemma 2.1]{Burk84}, and illustrates that $\beta_{p,X}^A$ can be expressed as a norm of a~certain martingale transform. Recall that $\textnormal{Ext}(A)$ is the set of all extreme points of a convex set $A$, i.e.\ all the points $a$ of $A$ such that $A \setminus \{a\}$ remains convex. Notice that $\text{Conv}(\text{Ext}(A)) = A$.

\begin{theorem}[$A$-martingale transform]\label{thm:Atransform}
  Let $1<p<\infty$, $X$ be a Banach space, $A\subset \mathbb K$ be closed, convex, and bounded. Let $(a'_n)_{n\geq 0}$ be a sequence with values in $\textnormal{Ext}(A)$ such that $(a'_n)_{n\geq N}$ is dense in $\textnormal{Ext}(A)$ for each $N\geq 0$. Consider a linear operator $T_A$ on $\mathcal M^{\rm dis}_p(X)$ such that for each $X$-valued $L^p$-martingale $f=(f_n)_{n\geq 0}$ one has that
 \[
  (T_A f)_n = a'_0f_0 + \sum_{i=1}^n a'_idf_i,\;\;\; n\geq 0.
 \]
Assume that there exists a sequence $(r_n)_{n\geq 1}$  of independent Rademacher random variables such that $r_n$ is $\mathcal F_n$-measurable and independent of $\mathcal F_{n-1}$ for each $n\geq 1$. Then $\|T_A\|_{\mathcal L(\mathcal M_p^{\rm dis}(X))}= \beta_{p,X}^A$.
\end{theorem}

\begin{remark}\label{rem:endpointsofA}
The natural example of $A$ is a closed convex polygon with vertices $a_1,\ldots, a_N \in \mathbb C$. Then the set of vertices is exactly the set of extreme points, and $(a'_n)_{n\geq 0}$ from Theorem \ref{thm:Atransform} can be taken as a periodic sequence involving all the numbers $a_1,\ldots, a_N$.
\end{remark}

\begin{proof}[Proof of Theorem \ref{thm:Atransform}]
Without loss of generality assume that $A$ contains at least two points, and that $X$ has the UMD property (which guarantees that $T_A \neq cI$ for some $c\in \mathbb K$ and that $0<\beta_{p,X}^A<\infty$; otherwise by (v) and (vii) from Proposition~\ref{prop:propertiesofUMP_p,X^A} the same machinery can be applied in the case of $T_A = cI$ or $\beta_{p,X}^A = \infty$). Fix $\delta>0$. By the definition of $\beta_{p,X}^A$ there exists a filtration $\mathbb G = (\mathcal G_m)_{m\geq 1}$, an $X$-valued $L^p$-integrable $\mathbb G$-martingale difference sequence $(d_m)_{m\geq 0}$, a sequence $(\eps_m)_{m=0}^M$ with values in $A$ such that
\begin{equation}\label{eq:Atransform}
  \Bigl(\mathbb E \Bigl\|\sum_{m=1}^M \eps_m d_m\Bigr\|^p\Bigr)^{\frac 1p}>(\beta_{p,X}^A - \delta)\Bigl(\mathbb E \Bigl\|\sum_{m=1}^M d_m\Bigr\|^p\Bigr)^{\frac 1p}.
\end{equation}
By \cite[Theorem 3.6.1]{HNVW1}, $(d_m)_{m\geq 0}$ can be chosen of the Paley-Walsh type. Since $A$ is convex, closed and bounded, there exists $J\geq 1$ for each $m=1,\ldots,M$ there exist $\lambda_{m,1},\ldots, \lambda_{m,J} \in [0,1]$, $\eps_{m,1},\ldots, \eps_{m,J} \in \textnormal{Ext}(A)$ such that $\eps_m = \sum_{j=1}^J\lambda_{m,j}\eps_{m,j}$ with $\sum_{j=1}^J\lambda_{m,j}=1$. Then by the triangle inequality we have that
\begin{align*}
  \Bigl(\mathbb E \Bigl\|\sum_{m=1}^M \eps_m d_m\Bigr\|^p\Bigr)^{\frac 1p} \leq \sum_{i_1,\ldots,i_M = 1}^J \prod_{k=1}^M\lambda_{k,i_k} \Bigl(\mathbb E \Bigl\|\sum_{m=1}^M \eps_{m,i_m} d_m\Bigr\|^p\Bigr)^{\frac 1p},
\end{align*}
and therefore due to the fact that 
$$
\sum_{i_1,\ldots,i_M = 1}^J \prod_{k=1}^M\lambda_{k,i_k} = \prod_{k=1}^M\sum_{i_1,\ldots,i_M = 1}^J \lambda_{k,i_k}=1
$$ 
and the inequality \eqref{eq:Atransform}, there exists a $\{1,\ldots,J\}$-valued sequence $(\tilde i_m)_{m=1}^M$ such that
\[
  \Bigl(\mathbb E \Bigl\|\sum_{m=1}^M \eps_{m,\tilde i_m} d_m\Bigr\|^p\Bigr)^{\frac 1p}>(\beta_{p,X}^A - \delta)\Bigl(\mathbb E \Bigl\|\sum_{m=1}^M d_m\Bigr\|^p\Bigr)^{\frac 1p}.
\]

Now by an approximation argument and the fact that $(a'_n)_{n\geq N}$ is dense in $\textnormal{Ext}(A)$ for each $N\geq 0$ we can find a subsequence $(a'_{n_m})_{m=1}^{M}$ such that $a'_{n_{m}} = \eps_{m,\tilde i_m}$ for each $m=1,\ldots,M$. Then by the special property of the filtration $(\mathcal F_n)_{n\geq 0}$ (which contains a Paley-Walsh filtration) there exists a~Paley-Walsh $X$-valued \mbox{$\mathbb F$-mar}\-tin\-ga\-le difference sequence $(\tilde d_n)_{n\geq 0}$ such that $\tilde d_{n_m}$ has the same distribution as $d_m$ for each $m=1,\ldots,M$, and $\tilde d_{n} = 0$ if $n\notin\{n_1,\ldots,n_M\}$. Let $(f_n)_{n\geq 0}$ be an $\mathbb F$-martingale such that $df_n = \tilde d_n$ for each $n\geq 0$. Then by \eqref{eq:Atransform} we have that 
$$
(\mathbb E\|(T_Af)_{\infty}\|^p)^{\frac 1p} >(\beta_{p,X}^A - \delta)(\mathbb E\|f_{\infty}\|^p)^{\frac 1p}.
$$ 
Therefore since $\delta$ was arbitrary, $\|T_A\|_{\mathcal L(\mathcal M_p^{\rm dis}(X))}\geq  \beta_{p,X}^A$. $\|T_A\|_{\mathcal L(\mathcal M_p^{\rm dis}(X))}\leq \beta_{p,X}^A$ follows from the definition of $\beta_{p,X}^A$.
\end{proof}

\begin{remark}
 Notice that for a fixed space $X$ and fixed $1<p<\infty$ one has that $(b,B)\mapsto \beta_{p,X}^{\{b,B\}}$ is a convex function in $(b,B)\in \mathbb R^2$. Indeed, due to Theorem \ref{thm:Atransform} together with Remark \ref{rem:endpointsofA} (and also by Example \ref{ex:a_1a_dalphain02}) we know that there exists a fixed Banach space $Y$ and a fixed operator $T\in \mathcal L(Y)$, depending only on $X$ and $p$, such that 
 $$
 \beta_{p,X}^{\{b,B\}} = \Bigl\|\frac{B-b}{2}T + \frac{b+B}{2}I\Bigr\|_{\mathcal L(Y)},
 $$
which implies the convexity in $(b,B)$-variable due to the convexity of the norm $\|\cdot\|_{\mathcal L(Y)}$ and the linearity of the operator $(b,B) \mapsto \frac{B-b}{2}T + \frac{b+B}{2}I$ as an operator from $\mathbb R^2$ to $\mathcal L(Y)$.

In particular, due to the symmetry of the function $c\mapsto \beta_{p,X}^{\{c- a,c+ a\}}$ we deduce that $ \beta_{p,X}^{\{c- a,c+ a\}}\geq  \beta_{p,X}^{\{- a,a\}}$ for each $a, c \in \mathbb R$ and that $ \beta_{p,X}^{\{c- a,c+ a\}}$ monotonically increases in $c\geq 0$ for each fixed $a\in \mathbb R$. 
\end{remark}

\subsection{$A$-Burkholder function and $A$-weak differential subordination}

In the current subsection we introduce the $A$-Burkholder function and $A$-weak differential subordination, and show the main results of the subsection, Theorem \ref{thm:AWDSofdiscretemartingales} and \ref{thm:bBwdsforpurdisc} (the latter one is the main tool for proving Theorem \ref{thm:multupperbound}). Throughout this subsection $A\subset \mathbb K$ is convex, bounded, closed, and contains at least two points. We will start with the following proposition.

\begin{proposition}\label{prop:Choi-Burkfunction}
 Let $X$ be a Banach space, $1<p<\infty$, $A\subset \mathbb K$. Then the following statements are equivalent:
 \begin{itemize}
  \item[(i)] $X$ is a UMD Banach space;
  \item[(ii)] there exist $\beta>0$ and $U:X\times X \to \mathbb R$ such that $z\mapsto U(x+z, y+\eps z)$ is a~concave function in $z\in X$ for each $x, y\in X$ and $\eps \in A$, and
  \[
   U(x,y)\geq \|y\|^p-\beta^p \|x\|^p,\;\;\; x,y\in X.
  \]
 \end{itemize}
 If this is the case, then the smallest admissible $\beta$ equals the UMD$_p^{A}$ constant $\beta_{p,X}^A$.
\end{proposition}

\begin{proof}
 The proof is analogous to the one presented in \cite[Theorem 4.5.6]{HNVW1} and~\cite{Y17FourUMD}, while instead of the Burkholder function we will construct its version which depends on $A$ as well. For each $x, y\in X$ we define $\mathbb S(x,y)$ as a set of all pairs $(f,g)$ 
of discrete martingales such that
 \begin{enumerate}
  \item $f_0\equiv x$, $g_0\equiv y$;
  \item there exists $N\geq 0$ such that $df_n\equiv 0$, $dg_n\equiv 0$ for 
$n\geq N$;
  \item $(dg_n)_{n\geq 1} = (\eps_ndf_n)_{n\geq 1}$ for some sequence of scalars 
$(\eps_n)_{n\geq 1}$ such that $\eps_n\in A$ for each $n\geq 1$.
 \end{enumerate}
Then we define $U:X\times X \to \mathbb R \cup \{+\infty\}$ as follows:
\begin{equation}\label{eq:formulaofU}
 U(x,y) := \sup\bigl\{\mathbb E\bigl(\|g_{\infty}\|^p - (\beta_{p,X}^A)^p \|f_{\infty}\|^p\bigr): 
(f,g)\in \mathbb S(x,y)\bigr\}.
\end{equation}
The rest of the proof repeats the one given in \cite[Theorem 4.5.6]{HNVW1}, except the proof of the fact that $U(x,y)<\infty$ for each $x,y\in X$. Let us show this separately.

First notice that $U(x, ax)\leq 0$ for each $x\in X$, $a\in A$, since $x\mapsto U(x,ax)$ is a~symmetric concave function such that $U(0,0)\leq 0$ (the latter holds due to the definition of $\beta_{p,X}^A$). Now fix $x,y\in X$ and prove that $U(x,y)<\infty$. For some fixed pair $a_1, a_2\in A$ define $z= \frac{a_2x-y}{a_2-a_1}$ and $u=\frac{a_1x-y}{a_1-a_2}$. Then since the mapping $w\mapsto U(z+w, a_1z+a_2w)$ is concave in $w\in X$, and due to the fact that $z+u = x$ and $a_1z+a_2u = y$,
\begin{equation}\label{eq:proofofboundednessofU(x,y)}
 \begin{split}
  U(x,y) + U(z-u, a_1z-a_2u) = U(z+u, a_1z+a_2u)&+ U(z-u, a_1z-a_2u)\\
  &\leq 2U(z, a_1z)\leq 0,
 \end{split}
\end{equation}
which means that both terms on the left hand side of \eqref{eq:proofofboundednessofU(x,y)} are finite.
\end{proof}

We will call the function $U$ defined by \eqref{eq:formulaofU} the {\it $A$-Burkholder function}, which coincides with the Burkholder function in the case of $A = \{\eps\in \mathbb K: |\eps|=1\}$ considered in \cite{Y17FourUMD} (the~Burkholder function, which is also known as the Bellman function, is widely used, see e.g.\ \cite{HNVW1}, \cite{Y17FourUMD} and \cite{Y17MartDec}). The following proposition demonstrates basic properties of the $A$-Burkholder function. Recall that if $E$ is a linear space over a scalar field $\mathbb K$, then a function $f:E\times E \to \mathbb R$ is called {\it biconcave} if both maps $x\mapsto f(x,y)$ and $x\mapsto f(y,x)$ are concave in $x\in E$ for each $y\in E$. Let us outline basic properties of the $A$-Burkholder function.

\begin{proposition}\label{prop:propertiesofUandV}
 Let $U:X\times X \to \mathbb R$ be as defined in \eqref{eq:formulaofU}. Then
 \begin{itemize}
  \item[(A)] $U(x,\cdot)$ is convex for each $x\in X$;
  \item[(B)] for any different $a_1, a_2\in A$ the function $V:X\times X \to \mathbb R$, defined in the following way
  \begin{equation}\label{eq:defofVChoi}
     V(x,y) = U\Bigl(\frac{x-y}{2},\frac{a_2x-a_1y}{2}\Bigr),\;\;\; x,y\in X,
  \end{equation}
is biconcave;
  \item[(C)] $U$ and $V$ are continuous;
  
  \item[(D)] if $X$ is finite dimensional, then $U$ and $V$ are a.s.\ Fr\'echet differentiable on $X\times X$;
  
  \item[(E)] the function $z\mapsto V\bigl(x+c_1z, y+c_2z\bigr)$ is concave in $z\in X$ for each $c_1, c_2 \in \mathbb K$ such that $\frac{a_2c_1-a_1c_2}{c_1-c_2}\in A$.
 \end{itemize}
\end{proposition}

\begin{proof}
 (A) follows from the definition \eqref{eq:formulaofU} of $U$, the fact that $x\mapsto \|x\|^p$ is a convex function on $X$, and the fact that the suprimum of convex functions is a~convex function.
 
 (B) holds due to the fact that $t\mapsto U(x+z, y+\eps z)$ is a~concave function in $z\in X$ for each $x, y\in X$ and $\eps \in \{a_1,a_2\}$.
 
 To prove (C) we will need to use the fact that $V$ is biconcave. By \cite[Proposition~3.2]{Thib84} any measurable biconcave function on $X\times X$ is continuous. Therefore we only need to prove that $V$ is measurable. To this end we have to show measurability of $U$, which follows from the fact that $U$ can be rewritten as a supremum over a countable family of continuous functions (since all the martingales in $\mathbb S(x,y)$ can be assumed to be Paley-Walsh by \cite[Theorem 3.6.1]{HNVW1}, and therefore by a density argument we can fix a countable subset of $\mathbb S(x,y)$). Hence $V$ is measurable due to the definition \eqref{eq:defofVChoi}, and consequently $V$ is continuous (we also recommend a~different proof of \cite[Proposition~3.2]{Thib84} presented in~\cite{AH86}, but this proof works only for finite dimensional spaces $X$). 
 
 To prove that $U$ is continuous it is sufficient to use the following ``back to $U$'' transform, which follows from the definition of $V$ given in \eqref{eq:defofVChoi}:
 \begin{equation}\label{VtoUtransform}
   U(x,y) = V\Bigl(\frac{2y-2xa_1}{a_2-a_1}, \frac{2y-2xa_2}{a_2-a_1}\Bigr),\;\;\; x,y\in X.
 \end{equation}
Since this transform is continuous linear, $U$ is continuous as a combination of a~continuous linear transform and a continuous function.

(Alternatively, one can show continuity of $U$ and $V$ without using such complicated techniques, e.g.\ by applying the fact that $U$ and $V$ are bounded from below and by generalizing \cite[Theorem 2.14]{BP12} to biconcave functions.)
 
 (D) can be shown analogously to the similar assertion from \cite{Y17FourUMD}: first we show the a.s.\ Fr\'echet differentiability of $V$ by applying \cite[Proposition 3.1]{JT}, and further we prove the same for $U$ using the ``back to $U$'' transform \eqref{VtoUtransform}
and applying the fact that this transform is linear.

(E) holds since by \eqref{eq:defofVChoi} and Proposition \ref{prop:Choi-Burkfunction} the function
\[
 z\mapsto V\bigl(x+c_1z, y+c_2z\bigr) = U\Bigl(\frac{x-y+z(c_1-c_2)}{2},\frac{a_2x-a_1y+z(a_2c_1-a_1c_2)}{2}\Bigr)
\]
is concave in $z\in X$ if $\frac{a_2c_1-a_1c_2}{c_1-c_2}\in A$.
\end{proof}

Let $X$ be a Banach space, $(f_n)_{n\geq 0}$ and $(g_n)_{n\geq 0}$ be two $X$-valued martingales. Then $(g_n)_{n\geq 0}$ is called {\it $A$-weakly differentially subordinated} to $(f_n)_{n\geq 0}$ (or simply $(g_n)_{n\geq 0}\stackrel{w,A}{\ll}(f_n)_{n\geq 0}$) if for each $x^*\in X^*$, $\frac{\langle dg_n, x^*\rangle}{\langle df_n, x^*\rangle}\in A$ a.s.\ for all $n\geq 1$ and $\frac{\langle g_0, x^*\rangle}{\langle f_0, x^*\rangle}\in A$ a.s., where we set $\frac 00\in A$ for any $A\subset \mathbb K$, and $\frac{c}{0}=\sign c \cdot\infty$ for any $c\neq 0$ (we wish to pay your attention that $\sign c$ here will not not play any r\^ole later).

\begin{theorem}
 Let $X$ be a Banach space, $A\subset \mathbb K$, $(f_n)_{n\geq 0}$ and $(g_n)_{n\geq 0}$ be two $X$-valued martingales such that $(g_n)_{n\geq 0}\stackrel{w,A}{\ll}(f_n)_{n\geq 0}$. Then there exists an adapted $A$-valued sequence $(a_n)_{n\geq 0}$ such that a.s.\ $g_0=a_0g_0$ and $dg_n = a_ndf_n$ for each $n\geq 1$. 
\end{theorem}

\begin{proof}
 The statement of the theorem is analogous to the similar assertion in \cite{Y17FourUMD}.
\end{proof}

\begin{theorem}\label{thm:AWDSofdiscretemartingales}
 Let $X$ be a Banach space. Then $X$ is a UMD Banach space if and only if for some (equivalently, for all) $1<p<\infty$ and $A\subset \mathbb K$ there exists a~constant $\beta>0$, depending only on $X$, $p$, and $A$, such that for all $X$-valued local martingales $(f_n)_{n\geq 0}$ and $(g_n)_{n\geq 0}$ with $(g_n)_{n\geq 0}\stackrel{w,A}{\ll}(f_n)_{n\geq 0}$ one has that
 \[
  \mathbb E \|g_n\|^p \leq \beta^p \mathbb E \|f_n\|^p,\;\;\; n\geq 0.
 \]
If this is the case, then the least admissible $\beta$ equals the UMD$_p^A$ constant $\beta_{p,X}^A$.
\end{theorem}

\begin{proof}
 The proof is similar to the one presented in \cite{Y17FourUMD}, but here one has to apply the $A$-Burkholder function from \eqref{eq:formulaofU} and Proposition \ref{prop:propertiesofUandV}.
\end{proof}

Now we turn to the continuous time case. Let $X$ be a Banach space. A martingale $M:\mathbb R_+\times \Omega \to X$ is called {\it purely discontinuous} if both $[\text{Re}\langle M, x^*\rangle]$ and $[\text{Im}\langle M, x^*\rangle]$ are pure jump processes for each $x^* \in X^*$ (where $[N]$ is a quadratic variation of a martingale $N:\mathbb R_+ \times \Omega \to \mathbb R$, see \cite{Kal,Y17FourUMD,Y17MartDec}). 

Let $1<p<\infty$, $M$, $N:\mathbb R_+ \times \Omega \to X$ be two purely discontinuous \mbox{$L^p$-mar}\-tin\-ga\-les. 
$N$ is {\it $A$-weakly differentially subordinated} to $M$ (or $N \stackrel{w,A}{\ll} M$) if for each $x^*\in X^*$ a.s.\ for each $t\geq 0$, $\frac{\langle \Delta N_t, x^*\rangle}{\langle \Delta M_t, x^*\rangle}\in A$, where again we set $\frac 00\in A$ for any $A\subset \mathbb K$, and $\frac{c}{0}=\sign c \cdot \infty$ for any $c\neq 0$, and where $\Delta M_t:= M_t - \lim_{\delta \to 0} M_{(t-\delta)\vee 0}$ exists since every $X$-valued local martingale has a c\`adl\`ag (continue \`a droite, limite \`a gauche) version (see~\cite{Y17FourUMD}).

\begin{remark}\label{rem:Awdsintermsofquadrvar}
Let $\mathbb K = \mathbb R$, $-\infty<b<B<\infty$, $A = \{b,B\}$. Then analogously to \cite[Theorem 1.6]{BO12} one can show that $N$ is $A$-weakly differentially subordinated to $M$ if and only if for each $x^* \in X^*$ a.s.
 \begin{equation}\label{eq:bBweakdiffsubdefcontimmart}
  \ud \Bigl[\Bigl\langle N-\frac{b+B}{2}M, x^*\Bigr\rangle\Bigr] \leq \ud \Bigl[\Bigl\langle\frac {B-b}{2}M, x^*\Bigr\rangle\Bigr].
\end{equation}
\end{remark}

 The following theorem is an extension of the similar one presented in \cite{Y17FourUMD}.

\begin{theorem}\label{thm:bBwdsforpurdisc}
 Let $X$ be a UMD Banach space, $1<p<\infty$, $A\subset \mathbb K$, $M$, $N:\mathbb R_+ \times \Omega\to X$ be purely discontinuous $L^p$-martingales such that $N \stackrel{w,A}{\ll} M$. Then for each $t\geq 0$
 \begin{equation}\label{eq:thm:bBwdsforpurdisc}
  \mathbb E \|N_t\|^p \leq (\beta_{p,X}^{A})^p \mathbb E \|M_t\|^p.
 \end{equation}
\end{theorem}

For the proof we will need the following lemma, which is analogous to the one given in \cite{Y17FourUMD}.

\begin{lemma}\label{lem:a(t,omega)forbBwds}
  Let $X$ be a UMD Banach space, $A\subset \mathbb K$, $M$, $N:\mathbb R_+ \times \Omega\to X$ be purely discontinuous local martingales such that $N \stackrel{w,A}{\ll} M$. Then there exists a set $\Omega_0$ of full measure such that for each $\omega \in \Omega_0$ and for each $t\geq 0$ the vectors $\Delta M_t(\omega)$ and $\Delta N_t(\omega)$ are collinear and there exists $a(t,\omega)\in A$ such that $\Delta N_t(\omega) = a(t,\omega)\Delta M_t(\omega)$.
\end{lemma}

\begin{proof}[Proof of Theorem \ref{thm:bBwdsforpurdisc}]
 The proof is analogous to \cite{Y17FourUMD}, but here one needs to apply Lemma~\ref{lem:a(t,omega)forbBwds}.
\end{proof}

\begin{remark}
It remains open what is a proper definition of $A$-weak differential subordination in the case of continuous martingales $M$ and $N$. From the one hand, it is more natural to set that $N\stackrel{w,A}{\ll} M$ if there exists a progressively measurable process $a:\mathbb R_+ \times \Omega \to A$ such that $N_t = a_0 M_0 +\int_0^t a_s\ud M_s$ in a weak sense, where the integral is well-defined due to the fact that $M$ is continuous. If this is the case, then the analogue of Theorem~\ref{thm:bBwdsforpurdisc} for continuous martingales follows immediately from an approximation of $a$ by simple processes, discretization and Theorem \ref{thm:AWDSofdiscretemartingales}.

On the other hand, such a definition at least in the case of $A = \{a\in \mathbb K:|a|=1\}$ does not seem to be useful (see Section~5 in \cite{Y17FourUMD}). A more useful definition, which extends the two given in \cite{BO12} and \cite{Y17MartDec}, can be characterized in the case of $\mathbb K = \mathbb R$ by \eqref{eq:bBweakdiffsubdefcontimmart}. But then it is unknown whether Theorem~\ref{thm:bBwdsforpurdisc} holds for general martingales $M$ and $N$ and such a definition of $A$-weak differential subordination, since it remains open whether the properties of the $A$-Burkholder function, which have been discussed and developed in \cite{BO12}, hold for any UMD Banach space $X$. Moreover, it remains unclear how to extend this definition to the case $\mathbb K = \mathbb C$.
\end{remark}

\section{Even Fourier multipliers}

It turns out that UMD$_p^A$ constants play an important r\^{o}le in Fourier multiplier theory. Namely, the norm of even Fourier multipliers in $L^p(\mathbb R^d;X)$ can be bounded by $\beta_{p, X}^{A}$ from below, where $A = \text{Ran}\bigl(m|_{\mathbb R^d\setminus \{0\}}\bigr)$ is the range of $m$. Furthermore, the norms of the so-called Ba\~{n}uelos-Bogdan multipliers can be bounded by $\beta_{p, X}^A$ from above, where $A$ is the range of the corresponding functions $\phi$ and $\psi$ from~\eqref{eq:defofm(xi)}. Notice that if $\text{Ran}\bigl(m|_{\mathbb R^d\setminus \{0\}}\bigr) \subset \mathbb R$ (resp.\ $\text{Ran}(\phi)\cup\text{Ran}(\psi)\subset \mathbb R$), then $X$ in Theorem~\ref{thm:multlowerbound} (resp.\ in Theorem \ref{thm:multupperbound}) can be considered over $\mathbb R$.

\smallskip

Let us recall some basic definitions. Let $d\geq 1$ be a natural number, $\mathcal S(\mathbb R^d)$~be a 
space of Schwartz functions on $\mathbb R^d$. For a Banach space $X$ with a 
scalar field $\mathbb C$ we define $\mathcal S(\mathbb R^d)\otimes X$ as the space 
of all functions $f:\mathbb R^d \to X$ of the form $f=\sum_{k=1}^K f_k \otimes 
x_k$, where $f_1,\ldots,f_K \in \mathcal S(\mathbb R^d)$ and 
$x_1,\ldots ,x_K \in X$. Notice that for each $1\leq p<\infty$ the space 
$\mathcal S(\mathbb R^d)\otimes X$ is dense in $L^p(\mathbb R^d; X)$.

Let $m:\mathbb R^d\to \mathbb C$ be measurable and bounded. We
define a linear operator $T_m$ acting on $\mathcal S(\mathbb R^d)\otimes X$ in the following way:
\begin{equation*}\label{eq:multSotimesX}
  T_m (f\otimes x) = \mathcal F^{-1}(m\mathcal F(f)) \cdot x,\;\;\;\; f\in 
\mathcal S(\mathbb R^d), x\in X,
\end{equation*}
where the {\it Fourier transform} $\mathcal F$ and the {\it inverse Fourier 
transform} $\mathcal F^{-1}$ are defined as follows:
\begin{align*}
 \mathcal F (f)(t) &= \frac{1}{(2\pi)^{d}}\int_{\mathbb R^d} e^{-i\langle 
t,u\rangle}f(u)\ud u,\;\;\; f\in \mathcal S(\mathbb R^d), t\in \mathbb R^d,\\
  \mathcal F^{-1} (f)(t) &= \int_{\mathbb R^d} 
e^{i\langle t,u\rangle}f(u)\ud u,\;\;\; f\in \mathcal S(\mathbb R^d), t\in 
\mathbb R^d.
\end{align*} 
The operator $T_m$ is called a {\it Fourier multiplier}, while the function $m$ 
is called the {\it symbol} of $T_m$. If $X$ is finite dimensional or Hilbert then $T_m$ can 
be extended to a bounded linear operator on $L^2(\mathbb R^d; X)$. The general question 
is whether one can extend $T_m$ to a bounded operator on $L^p(\mathbb 
R^d;X)$ for a~general $1<p<\infty$ and a given space $X$ and what is then the corresponding norm of $T_m$. Here we show sharp lower and upper bounds for $\|T_m\|$ for special classes of $m$ and $X$ with the UMD property in terms of UMD$_p^A$ constants of $X$.

\smallskip

In this section we set $\frac{c}{0}=0$ for each $c\in \mathbb C$.

\subsection{Lower bounds}

The following theorem is a variation of \cite[Proposition~3.4]{GM-SS}. Recall that function $f:\mathbb R^d \to \mathbb C$ is homogeneous if $f(cx) = f(x)$ for each~$c>0$ and $x\in \mathbb R^d$.

\begin{theorem}\label{thm:multlowerbound}
 Let $1<p<\infty$, $X$ be a Banach space, $m:\mathbb R^d \to \mathbb C$ be an even homogeneous function such that $m|_{\mathbb R^d\setminus \{0\}}$ is continuous. Let $A := \textnormal{Ran}\bigl(m|_{\mathbb R^d\setminus \{0\}}\bigr)\subset \mathbb C$.  Then we have that $\|T_m\|_{\mathcal L(L^p(\mathbb R^d;X))}\geq \beta_{p,X}^{A}$. In particular, if $X$ is not a UMD space and $m$ is not a constant, then $T_m$ is unbounded on $L^p(\mathbb R^d;X)$.
\end{theorem}

For the proof we will need several intermediate steps.
\begin{lemma}\label{lemma:linoperofsymbol}
 Let $S\in \mathcal L(\mathbb R^d)$ be a linear bounded invertible operator. Then 
 $$
 \|T_{m}\|_{\mathcal L(L^p(\mathbb R^d;X))} = \|T_{m\circ S}\|_{\mathcal L(L^p(\mathbb R^d;X))}.
 $$
\end{lemma}

\begin{proof}
 This follows from \cite[(7)]{GM-SS}.
\end{proof}

Let $\mathbb T := (-\pi,\pi]$, $\widetilde m:\mathbb Z^d\to \mathbb C$. Then we define Fourier multiplier $T_{\widetilde m}$ acting on a function $f:\mathbb T^d \to X$ in the following way
\[
 (T_{\widetilde m} f)(\theta):= \sum_{k\in \mathbb Z^d}\hat{f}(k)e^{i\langle k,\theta\rangle}\widetilde m(k),\;\;\; \theta \in \mathbb T^d,
\]
where 
$$
\hat f(k) := \frac{1}{(2\pi)^d}\int_{\mathbb T^d}e^{-i\langle k,\theta\rangle}f(\theta) \ud \theta,\;\;\; k\in \mathbb Z^d.
$$
The following lemma demonstrates that if $m(k) = \widetilde m(k)$ for each $k\in \mathbb Z^d$, then $\|T_m\| = \|T_{\widetilde m}\|$; it was initially shown in the scalar-valued case in \cite{deL65}, and in the vector-valued case in \cite{Graf} and in \cite{GM-SS}. A simple proof can be found in \cite[Corollary~5.7.6]{HNVW1}.

\begin{lemma}\label{lem:deLeeuw}
 Let $X$ be a Banach space, $1<p<\infty$, $m:\mathbb R^d\to \mathbb C$ be homogeneous such that $m|_{\mathbb R^d\setminus \{0\}}$ is continuous.
Let $\widetilde m:\mathbb Z^d \to \mathbb C$ be such that $\widetilde m(k) = m(k)$ for each $k\in \mathbb Z^d$, where
\[
 \widetilde m(0)= m(0) := \frac{1}{|B(0,1)|}\int_{B(0,1)} m(s)\ud s.
 \] 
 Then
\[
 \|T_m\|_{\mathcal L(L^p(\mathbb R^d;X))} = \|T_{\widetilde m}\|_{\mathcal L(L^p(\mathbb T^d;X))}.
\]
\end{lemma}

The following lemma is similar to \cite[Lemma 3.3]{GM-SS}.

\begin{lemma}\label{lem:multidimcasededuction}
 Let $1<p<\infty$, $X$ be a Banach space, $m:\mathbb R^d\to \mathbb C$ be homogeneous such that $m|_{\mathbb R^d\setminus \{0\}}$ is continuous, and let $Q = \mathbb T^d$. For each $k\geq 1$ define $E_k$ as the closure of the finite trigonometric polynomials
 \[
  \Phi_k(\theta_1,\ldots,\theta_k) = \sum_{j=1}^J \Phi^j_k(\theta_1,\ldots,\theta_k)x_p,
 \]
with $x_k\in X$ and 
\[
 \Phi_k^j = \sum_{\ell_1\in \mathbb Z^d}\cdots \sum_{\ell_k \in \mathbb Z^d}e^{i\langle \ell_1,\theta_1\rangle}\cdots e^{i\langle \ell_k,\theta_k\rangle}\alpha_{\ell_1,\ldots,\ell_k}^j,
\]
where only finitely many $\alpha_{\ell_1,\ldots,\ell_k}^j \in \mathbb C$ are nonzero, and $\alpha_{\ell_1,\ldots,\ell_k}^j = 0$ if $\ell_k=0$. Let $T^k_{\widetilde m}:L^p(Q^k)\to L^p(Q^k)$ be given by
\[
 (T^k_{\widetilde m}\Phi_k)(\theta_1,\ldots,\theta_k) = \sum_{j=1}^J \Bigl( \sum_{\ell_1\in \mathbb Z^d}\cdots \sum_{\ell_k \in \mathbb Z^d}m(\ell_k)e^{i\langle \ell_1,\theta_1\rangle}\cdots e^{i\langle \ell_k,\theta_k\rangle}\alpha_{\ell_1,\ldots,\ell_k}^j\Bigr)x_p
\]
for all $\ell_1,\ldots,\ell_k\in\mathbb Z^d$. Then one has that
\[
 \Biggl\| \sum_{k=1}^n(T^k_{\widetilde m}\Phi_k)(\theta_1,\ldots,\theta_k) \Biggr\|_{L^p(Q^n;X)} \leq \|T_{\widetilde m}\|_{\mathcal L(L^p(\mathbb T^d;X))}\Biggl\| \sum_{k=1}^n\Phi_k(\theta_1,\ldots,\theta_k) \Biggr\|_{L^p(Q^n;X)}.
\]
\end{lemma}

\begin{proof}
 The proof is essentially the same as the one of \cite[Lemma 3.3]{GM-SS}, but to cover all continuous symbols $m$ (instead of $m|_{\mathbb R^d\setminus \{0\}}\in C^1$, as was assumed in the original proof), we need to change the last step. Therefore we will not repeat the whole original proof, but only present the revised last step.
 
 Let $\mathcal L^k\subset \mathbb Z^d$ be the set of all $\ell \in \mathbb Z^d$ such that $\alpha^j_{\ell_1,\ldots,\ell_{k-1},\ell}\neq 0$ for some $j\in \{1,\ldots,J\}$ and $\ell_1,\ldots,\ell_{k-1}\in\mathbb Z^d$. Then $0\notin \mathcal L^k$ and $\mathcal L^k$ is bounded. Now let $\widetilde{\mathcal L}^k\subset \mathbb Z^d$ be such that $\widetilde{\mathcal L}^k =\cup_{\alpha_{\ell_1,\ldots,\ell_k}^j \neq 0}(\ell_1\cup\cdots\cup \ell_k)$. Then $\widetilde{\mathcal L}^k$ is bounded as well. Let $R :=\sup_{\ell\in \widetilde{\mathcal L}^k}\|\ell\|$. In the end of the proof of \cite[Lemma 3.3]{GM-SS} one needed to show that $|m(\ell_1 A^{-k+1} + \ldots+ \ell_{k-1}A^{-1} + \ell) - m(\ell)| \to 0$ as $A\to \infty$ for each fixed $\ell_1,\ldots,\ell_{k-1}\in  \widetilde{\mathcal L}^k$ uniformly in $\ell \in \mathcal L^k$. Without loss of generality assume that $A>2k^2R^2$. Let $\mathcal O^k \subset \mathbb R^d$ be as follows: $\mathcal O^k := \mathcal L^k + \frac{1}{kR}\widetilde{\mathcal L}^k$. Then $\mathcal O^k$ is a compact that does not contain $\{0\}$, and therefore $m|_{\mathcal O^k}$ is continuous, hence uniform continuous. Therefore, since 
 $$
 \|\ell_1 A^{-k+1} + \ldots+ \ell_{k-1}A^{-1}\| \leq A^{-1}kR,\;\;\; \ell_1,\ldots,\ell_{k-1}\in  \widetilde{\mathcal L}^k,
 $$
 $|m(\ell_1 A^{-k+1} + \ldots+ \ell_{k-1}A^{-1} + \ell) - m(\ell)| \to 0$ as $A\to \infty$ uniformly in $\ell \in \mathcal L^k$ for each $\ell_1,\ldots,\ell_{k-1}\in  \widetilde{\mathcal L}^k$ by uniform continuity of $m$ on $\mathcal O^k$.
\end{proof}

\begin{lemma}\label{lem:extenddimforFourmult}
 Let $X$ be a Banach space, $1<p<\infty$, $d_2>d_1$, $m_1:\mathbb R^{d_1}\to \mathbb C$, $m_2:\mathbb R^{d_2}\to \mathbb C$ be such that 
 $$
 m_2(x_1,\ldots,x_{d_2}) = m_1(x_1,\ldots,x_{d_1})\;\; \text{for each}\;\; (x_1,\ldots,x_{d_2})\in \mathbb R^{d_2}.
 $$
 Then $\|T_{m_1}\|_{\mathcal L(L^p(\mathbb R^{d_1};X))} = \|T_{m_2}\|_{\mathcal L(L^p(\mathbb R^{d_2};X))}$.
\end{lemma}

\begin{proof}
The proof is analogous to the one given in \cite[Theorem 2.5.16]{Graf}.
\end{proof}

\begin{proof}[Proof of Theorem \ref{thm:multlowerbound}]
 The given proof is the modification of the one of \cite[Theorem~1.4]{BO12}. First fix $\{a_1,\ldots,a_K\} := A_K \subset A$ and show that \
 $$
 \|T_{m}\|_{\mathcal L(L^p(\mathbb R^{d};X))}\geq \beta_{p,X}^{A_K}.
 $$ 
 By Lemma \ref{lem:extenddimforFourmult} we can extend the dimension and assume that $d:= d+K$ by adding $K$ coordinates so that $m$ does not depend on the last $K$ coordinates. Since $m|_{\mathbb R^d\setminus \{0\}}$ is homogeneous and continuous, and by the fact that the unit sphere in $\mathbb R^d$ is compact, there exist unit vectors $e_1,\ldots,e_K$ in $\mathbb R^d$ such that $m(e_k)=a_k$ for each $k=1,\ldots,K$. Without loss of generality we may assume that $(e_k)_{k=1}^K$ are linearly independent (otherwise we can change the last $K$ coordinates of the vectors $(e_k)_{k=1}^K$ to make them linearly independent, and later renormalize them). By Lemma \ref{lemma:linoperofsymbol} we can find an~appropriate transform $S$ of $\mathbb R^d$ such that $(e_k)_{k=1}^{K}$ will be the first $K$ basis vectors of $\mathbb R^d$, i.e.\ $(m\circ S)(e_k) = a_k$ for each $k=1,\ldots,K$.
 
 Fix $\eps>0$. By \eqref{eq:UMD_p^Aconstant} and \cite[Theorem 3.6.1]{HNVW1} there exists an $X$-valued Paley-Walsh martingale $(f_n)_{n=0}^N$ and coefficients $(b_n)_{n=0}^N$ with values in $A_K$ such that
 \begin{equation}\label{eq:prooflowerboundsPWmart}
    \mathbb E \Bigl\|\sum_{n =0}^N b_ndf_n\Bigr\|^p\geq (\beta_{p, X}^{A_K}-\eps)^p  \mathbb E \|f_N\|^p.
 \end{equation}
Let $(r_n)_{n=1}^N$ be a Rademacher sequence, $\phi_n:\{-1,1\}^{n-1} \to X$ for each $n=1,\ldots,N$ be such that $df_n = r_n \phi(r_1,\ldots,r_{n-1})$ for each $n=1,\ldots,N$.

Now let $Q = \mathbb T^d$. For each $n=1,\ldots,N$ define $R_n:Q \to \mathbb R$ as follows: $R_n(\theta_{1},\ldots,\theta_{d}) = \text{sign}(\theta_{s(n)})$, where $s(n)$ is such that $s(n) = k$ if $b_n = a_k$. Define function $f:Q^N \to X$ as follows:
\[
 f(\theta^1,\ldots,\theta^N) = \sum_{n=0}^N R_n(\theta^n)\phi_n(R_1(\theta^1),\ldots,R_{n-1}(\theta^{n-1})),
\]
for each $\theta^1,\ldots,\theta^N \in Q$. Notice that $(T_{\widetilde m}R_n)(\theta) = b_n R_n(\theta)$ for each $n=1,\ldots,N$ due to the fact that $R_n(\theta)$ depends only on $\theta_{s(n)}$, $\int_{Q}R_n(\theta) \ud \theta = 0$, and the fact that $\widetilde m(j e_{s(n)}) = b_n$ for any $j\in \mathbb Z\setminus \{0\}$. Therefore by Lemma \ref{lem:deLeeuw} and \ref{lem:multidimcasededuction}
\begin{multline}\label{eq:lowerboundstroughQ^N}
 \|T_m\|_{\mathcal L(L^p(\mathbb R^{d};X))} = \|T_{\widetilde m}\|_{\mathcal L(L^p(\mathbb T^{d};X))}\\
 \geq \frac{\|\sum_{n=1}^Nb_n R_n(\theta^n)\phi_n(R_1(\theta^1),\ldots,R_{n-1}(\theta^{n-1}))\|_{L^p(Q^N;X)}}{\|\sum_{n=1}^N R_n(\theta^n)\phi_n(R_1(\theta^1),\ldots,R_{n-1}(\theta^{n-1}))\|_{L^p(Q^N;X)}}.
\end{multline}
Now notice that after renormalization of $Q^N$, $(R_n(\theta^n))_{n=1}^N$ are independent Rademacher's, and therefore
\begin{align*}
 \Bigl\|\sum_{n=1}^Nb_n R_n(\theta^n)\phi_n(R_1(\theta^1),\ldots,R_{n-1}(\theta^{n-1}))\Bigr\|_{L^p(Q^N;X)}^p &= \mathbb E \Bigl\|\sum_{n =0}^N b_ndf_n\Bigr\|^p,\\
 \Bigl\|\sum_{n=1}^N R_n(\theta^n)\phi_n(R_1(\theta^1),\ldots,R_{n-1}(\theta^{n-1}))\Bigr\|_{L^p(Q^N;X)}^p &= \mathbb E \|f_N\|^p,
\end{align*}
and so by \eqref{eq:prooflowerboundsPWmart} and \eqref{eq:lowerboundstroughQ^N}
\[
 \|T_m\|_{\mathcal L(L^p(\mathbb R^{d};X))} \geq \beta_{p,X}^{A_N} - \eps.
\]
Since $\eps >0$ was arbitrary and due to Proposition \ref{prop:propertiesofUMP_p,X^A}(x) (we can choose $(A_N)_{N\geq 1}$ such that $\text{Conv}(A_N)\nearrow \overline{\text{Conv}(A)}$ as $N\to \infty$ by choosing a dense sequence $(a_n)_{n=1}^{\infty}$ in $A$), the desired holds.
\end{proof}

Our first corollary will be about the so-called {\it Beurling-Ahlfors transform}, which is denoted by $\mathcal {BA}$ and is defined as follows:
\[
 (\mathcal {BA}\; f)(z) = -\frac{1}{\pi}\int_{\mathbb C}\frac{f(w)}{(z-w)^2}\ud w,\;\;\; z\in \mathbb C\simeq \mathbb R^2.
\]

\begin{corollary}\label{cor:BAtranslowerbounds}
 Let $X$ be a Banach space, $1<p<\infty$. Then 
 $$
 \|\mathcal {BA}\|_{\mathcal L(L^p(\mathbb R^2;X))}\geq \beta_{p,X}^{\mathbb D}.
 $$
\end{corollary}

\begin{proof}
 It is well-known (see \cite[p.\ 493]{HNVW1}) that $\mathcal {BA}$ is a Fourier multiplier with a~symbol $m(z) = \frac{\bar{z}}{z}$, $z\in\mathbb C \simeq \mathbb R^2$. Therefore $\text{Conv}(\text{Ran}(m)) = \mathbb D$, and the statement follows from Theorem \ref{thm:multlowerbound} and Proposition \ref{prop:propertiesofUMP_p,X^A}(ii).
\end{proof}

\begin{remark}\label{rem:BAtranssharpbounds}
 The norm of the Beurling-Ahlfors transform even in the scalar-valued case remains an open problem (see Iwaniec conjecture e.g.\ in \cite[Problem~O.1]{HNVW1}). Corollary \ref{cor:BAtranslowerbounds} together with \cite[Corollary 3.2]{GM-SS} gives us the following general estimates
 \[
  \beta_{p, X}^{\mathbb D} \leq \|\mathcal {BA}\|_{\mathcal L(L^p(\mathbb R^2;X))}\leq 2\beta_{p,X}^{\{-1,1\}},
 \]
which is a semiextension (up to the conjecture on equality of $\beta_{p, X}^{\mathbb D}$ and $\beta_{p,X}^{\{-1,1\}}$, see Remark \ref{rem:beta-11vsbetaD}) of the known due to \cite{GM-SS} estimate
\[
\beta_{p,X}^{\{-1,1\}}\leq \|\mathcal {BA}\|_{\mathcal L(L^p(\mathbb R^2;X))} \leq2\beta_{p,X}^{\{-1,1\}}.
\]
\end{remark}

\subsection{Upper bounds}

In the current subsection we will show sharp upper bounds for Fourier multipliers with symbols of the form \eqref{eq:defofm(xi)}, which generalize \cite[Theorem~1.7]{BO12} to the UMD space case and give more precise estimates for the results from \cite{Y17FourUMD}, which now depend on the range of $\phi$ and $\psi$ from \eqref{eq:defofm(xi)}.

Let $V$ be a L\'evy measure on $\mathbb R^d$, that is $V(\{0\})=0$, $V \neq 0$ 
and
$$
\int_{\mathbb R^d}(|x|^2\wedge 1) V(\dd x) <\infty.
$$
Let $\mu\geq 0$ be a 
finite Borel measure on the unit sphere $S^{d-1} \subset \mathbb R^d$. Finally let $A\subset \mathbb C$ be bounded, closed, and convex (these assumptions do not restrict the generality), $\phi \in L^{\infty}( \mathbb R^d ; \mathbb C)$ and
$\psi\in L^{\infty}( S^{d-1}; \mathbb C)$ be with values in $A$. The symbol of the subsection is
 \begin{equation}\label{eq:defofm(xi)}
  m(\xi) = \frac{\int_{\mathbb R^d} (1-\cos (\xi \cdot z))\phi(z)V(\dd z) + 
\frac 12\int_{S^{d-1}}(\xi \cdot \theta)^{2}\psi(\theta)\mu(\dd\theta)}{\int_{\mathbb 
R^d} (1-\cos (\xi \cdot z)z)V(\dd z)+ \frac 12\int_{S^{d-1}}(\xi \cdot 
\theta)^{2}\mu(\dd\theta)},\;\; \xi \in \mathbb R^d,
 \end{equation}
where we assume $\frac{c}{0}=0$ for each $c\in\mathbb C$.

\begin{theorem}\label{thm:multupperbound}
 Let $X$ be a UMD Banach space, $A \subset \mathbb C$, $\phi \in L^{\infty}(\mathbb R^d)$ and $\psi \in L^{\infty}(S^{d-1})$ be with values in $A$. Then the Fourier multiplier $T_m$ with a~symbol $m$ given in \eqref{eq:defofm(xi)}
 has a bounded extension on $L^p(\mathbb R^d; X)$ for $1<p<\infty$. Moreover, 
then for each $f\in L^p(\mathbb R^d; X)$
 \begin{equation}\label{eq:multthm}
   \|T_m f\|_{L^p(\mathbb R^d;X)}\leq \beta_{p,X}^A\|f\|_{L^p(\mathbb R^d;X)}.
 \end{equation}
\end{theorem}

\begin{proof}
 First of all symmetrize $\phi$, $\psi$, $V$, and $\mu$ as it was done in \cite{BBB}. Notice that the new $\phi^*$ and $\psi^*$ still take values in $A$, since they are convex combinations of the former $\phi$ and $\psi$ (for instance, $\phi^*(\xi) = \frac{1+k(\xi)}{2}\phi(\xi) + \frac{1-k(\xi)}{2}\phi(-\xi)$ for some $k(\xi) \in [-1,1]$ for each $\xi \in \mathbb R^d$). The rest of the proof is analogous to the ones given in \cite{Y17FourUMD}, \cite{BB}, and \cite{BBB}, but one needs to use $A$-weak differential subordination, namely Theorem \ref{thm:bBwdsforpurdisc}, instead of weak differential subordination.
\end{proof}

\subsection{Examples}

Now we will present some examples of Fourier multipliers with symbols of the form \eqref{eq:defofm(xi)}. It turns out that Theorem \ref{thm:multlowerbound} is applicable for many of them.

\begin{example}
The following examples of Fourier multiplier satisfying \eqref{eq:defofm(xi)} are similar to the ones given in \cite{BB,BBB,BO12,Y17FourUMD}. Let $\psi\in L^{\infty}(S^{d-1})$ be with values in $A\subset \mathbb C$. Then for both symbols

  \begin{equation}\label{eq:exampletothepoweralpha}
    m(\xi) =\frac{ \int_{S^{d-1}}|(\xi \cdot 
\theta)|^{\alpha}\psi(\theta)\mu(\dd\theta)}{\int_{S^{d-1}}|(\xi \cdot 
\theta)|^{\alpha}\mu(\dd\theta)}, \;\;\; \xi \in \mathbb R^d,\;\;\alpha \in (0,2],
  \end{equation}
  
   \[
  m(\xi)= \frac{\int_{S^{d-1}} \ln(1+(\xi\cdot \theta)^{-2})\psi(\theta)\mu(\dd 
\theta)}{\int_{S^{d-1}} \ln(1+(\xi\cdot \theta)^{-2})\mu(\dd \theta)},\;\;\;\xi 
\in \mathbb R^d,
 \]
 we have that $\|T_m\|_{\mathcal L(L^p(\mathbb R^d;X))}\leq \beta_{p,X}^A$.
\end{example}

\begin{example}{\cite[(1.27)]{BO12}}
 Let $d=2n$, $\alpha \in (0,2]$. If
 \[
  m(\xi) = \frac{|\xi_1^2 +\cdots +\xi_n^2|^{\alpha/2}}{|\xi_1^2 +\cdots +\xi_n^2|^{\alpha/2} + |\xi_{n+1}^2 +\cdots +\xi_{2n}^2|^{\alpha/2}},\;\;\; \xi \in \mathbb R^d,
 \]
then $\|T_{m}\|_{\mathcal L(L^p(\mathbb R^d;X))} = \beta_{p,X}^{\{0,1\}}$ (``$\leq$'' follows from Theorem \ref{thm:multupperbound}, ``$\geq$'' follows from Theorem \ref{thm:multlowerbound}).
\end{example}

\begin{example}\label{ex:a_1a_dalphain02}
 Let $\alpha \in (0,2]$, $a_1,\ldots,a_d\in \mathbb C$. Let $(e_j)_{j=1}^d$ be an orthonormal basis of $\mathbb R^d$, $\mu = \delta_{e_1}+ \ldots + \delta_{e_d}$, $\psi(e_j) = a_j$ for each $j=1,\ldots,d$. Then $m$ from \eqref{eq:exampletothepoweralpha} becomes
 \[
  m(\xi) = \frac{a_1|\xi_1|^{\alpha} + \cdots + a_d|\xi_d|^{\alpha}}{|\xi_1|^{\alpha} + \cdots + |\xi_d|^{\alpha}},\;\;\; \xi \in \mathbb R^d,
 \]
and one has $\|T_m\|_{L^p(\mathbb R^d;X)} = \beta_{p,X}^{\{a_1,\ldots,a_d\}}$ (``$\leq$'' follows from Theorem \ref{thm:multupperbound}, while ``$\geq$'' follows from Theorem \ref{thm:multlowerbound}). In particular $\|a_1 R_1^2 + \cdots + a_d R_d^2\|_{L^p(\mathbb R^d;X)} = \beta_{p,X}^{\{a_1,\ldots,a_d\}}$, where $R_j$ is the corresponding Riesz transform.

This also means that the bounds provided in both Theorem~\ref{thm:multlowerbound} and Theorem~\ref{thm:multupperbound} are sharp. 
\end{example}

\begin{example}
 Let $\alpha \in (0,2]$, $c\geq 0$. Let 
 \[
  m_c(\xi) = \frac{|\xi_1|^{\alpha}}{c + |\xi_1|^{\alpha} + \cdots + |\xi_d|^{\alpha}},\;\;\; \xi \in \mathbb R^d.
 \]
Then $\|T_{m_c}\|_{\mathcal L(L^p(\mathbb R^d;X))}=\beta_{p,X}^{\{0,1\}}$. To show this first notice that by Example \ref{ex:a_1a_dalphain02} we may assume that $c>0$, and that then due to Lemma \ref{lemma:linoperofsymbol} we have that $\|T_{m_c}\|_{\mathcal L(L^p(\mathbb R^d;X))}$ does not depend on $c$. Let
\[
 M(\xi) =  \frac{|\xi_1|^{\alpha}}{|\xi_1|^{\alpha} + \cdots + |\xi_{d+1}|^{\alpha}},\;\;\; \xi \in \mathbb R^{d+1}.
\]
Then by Example \ref{ex:a_1a_dalphain02} $\|T_{M}\|_{\mathcal L(L^p(\mathbb R^{d+1};X))}=\beta_{p,X}^{\{0,1\}}$, and hence by \cite[Theorem 2.5.16]{Graf} for a.e.\ fixed $\xi_{d+1}\in \mathbb R$
\[
 \|T_{m_{|\xi_{d+1}|^{\alpha}}}\|_{\mathcal L(L^p(\mathbb R^d;X))} \leq \|T_{M}\|_{\mathcal L(L^p(\mathbb R^d;X))}=\beta_{p,X}^{\{0,1\}},
\]
so $\|T_{m_c}\|_{\mathcal L(L^p(\mathbb R^d;X))}\leq\beta_{p,X}^{\{0,1\}}$ for each $c>0$ since the norm does not depend on~$c$. On the other hand by Example \ref{ex:a_1a_dalphain02} and analogously to \cite[Proposition 5.5.2]{HNVW1}, 
$$
\beta_{p,X}^{\{0,1\}} \leq \|T_{m_0}\|_{\mathcal L(L^p(\mathbb R^d;X))}\leq \|T_{m_c}\|_{\mathcal L(L^p(\mathbb R^d;X))}.
$$
\end{example}

\begin{example}\label{ex:rho(alpha)in02}
 Let $\rho$ be a probability measure on $(0,2]$, $a_1,\ldots,a_d\in \mathbb C$. If
 \begin{equation}\label{eq:symbmwithprmeasure}
  m(\xi) = \int_{\alpha \in (0,2]}\frac{a_1|\xi_1|^{\alpha} + \cdots + a_d|\xi_d|^{\alpha}}{|\xi_1|^{\alpha} + \cdots + |\xi_d|^{\alpha}}\rho(\dd \alpha),\;\;\; \xi \in \mathbb R^d,
 \end{equation}
then $\|T_m\|_{L^p(\mathbb R^d;X)} = \beta_{p,X}^{\{a_1,\ldots,a_d\}}$: ``$\leq$'' follows from Example \ref{ex:a_1a_dalphain02} and the triangle inequality, and ``$\geq$'' follows from Theorem \ref{thm:multlowerbound}.
\end{example}

\begin{remark}\label{rem:addingcItothesymbol}
 Notice that for each $c\in \mathbb C$
 \[
  \frac{(a_1+c)|\xi_1|^{\alpha} + \cdots + (a_d+c)|\xi_d|^{\alpha}}{|\xi_1|^{\alpha} + \cdots + |\xi_d|^{\alpha}} = \frac{a_1|\xi_1|^{\alpha} + \cdots + a_d|\xi_d|^{\alpha}}{|\xi_1|^{\alpha} + \cdots + |\xi_d|^{\alpha}} + c,\;\;\; \xi\in \mathbb R^d\setminus \{0\}.
 \]
Therefore for $m$ defined by \eqref{eq:symbmwithprmeasure}
\[
 \|T_m+cI\|_{L^p(\mathbb R^d;X)} = \|T_{m+c}\|_{L^p(\mathbb R^d;X)} = \beta_{p,X}^{\{a_1+c,\ldots,a_d+c\}}.
\]
\end{remark}

\begin{example}\label{exFourmultGroenevelt}
 Let $d=2$, $0\leq u<v\leq 2$, $\rho$ be the uniform distribution on $(u,v]$. Define $\kappa:(0,\infty)\to\mathbb R$ in the following way
 \[
  \kappa(t):= \frac{\log(\frac{t^u+1}{t^v+1})}{\log(\frac{t^u}{t^v})},\;\;\; t>0.
 \]
Then by Example \ref{ex:rho(alpha)in02} we have that if
\[
    m(\xi) = a_1 \Biggl(1+\kappa\left(\frac{|\xi_2|}{|\xi_1|}\right)\Biggr)
    + a_2 \Biggl(1+\kappa\left(\frac{|\xi_1|}{|\xi_2|}\right)\Biggr),\;\;\; \xi \in \mathbb R^d,
\]
then $\|T_m\|_{L^p(\mathbb R^d;X)}= \beta_{p,X}^{\{a_1,a_2\}}$. In particular if $a_1=1$, $a_2=0$, and $u=0$, then
\[
  m(\xi) =1+\frac{\log{2}-\log{\bigl(\frac{|\xi_2|^v}{|\xi_1|^v}+1\bigr)}}{\log{\bigl(\frac{|\xi_2|^v}{|\xi_1|^v}\bigr)}},\;\;\; \xi \in \mathbb R^d,
\]
is such that $\|T_m\|_{L^p(\mathbb R^d;X)}= \beta_{p,X}^{\{0,1\}}$. Moreover, by Remark \ref{rem:addingcItothesymbol} one gets that $\|T_{m}+cI\|_{L^p(\mathbb R^d;X)} = \|T_{m+c}\|_{L^p(\mathbb R^d;X)} = \beta_{p,X}^{c,c+1}$ for any $c\in \mathbb C$.
\end{example}

\begin{remark}\label{rem:conterexevenhomandunbdd} The reader might think that any even homogeneous Fourier multiplier is of the form \eqref{eq:defofm(xi)} for some $\phi$ and $\psi$ such that $\|\phi\|_{\infty}, \|\psi\|_{\infty}<\infty$. Unfortunately this is not true, and there are bounded even homogeneous Fourier multipliers which are unbounded on $L^p(\mathbb R^d)$ for any $p\neq 2$. Nonconstructively this fact was shown in \cite[pp.\ 59-60]{CS95} in the case $d\geq 3$. Here we construct such a symbol in the general case $d\geq 2$. Let
\[
m(\xi) =
\begin{cases}
e^{i\frac{|\xi_1|^2 + \cdots + |\xi_d|^2}{|\xi_d|^2}},\;\;\; &\xi_1\neq 0,\\
 0,\;\;\;&\xi_1= 0.
\end{cases}                                                                                                                                                                                                                                                                                                                                                                                                                                                                                                                                                                                                                                                                                                                                                                                                                                                                                                                          \]
Fix $p\neq 2$ and assume that $\|T_m\|_{\mathcal L(L^p(\mathbb R^d))}<\infty$. Then by \cite[Theorem 2.5.16]{Graf} for a.e.\ fixed $\xi_d \in \mathbb R\setminus\{0\}$ a multiplier with a symbol $e^{i\frac{|\xi_1|^2 + \cdots + |\xi_d|^2}{|\xi_d|^2}}$ is bounded on $L^p(\mathbb R^{d-1})$. But this contradicts with the H\"ormander theorem (see Lemma 1.4 and Theorem 1.14 in \cite{Hor60}). Thus $\|T_m\|_{\mathcal L(L^p(\mathbb R^d))}=\infty$ for any $p\neq 2$.
\end{remark}

\begin{remark}
 Let $p\in (1,2)\cup (2,\infty)$, $d\geq 2$. In this case any Fourier multiplier $T_m$ with $C^{\infty}$ even homogeneous symbol $m$ is bounded on $L^p(\mathbb R^d;X)$ with $X$ having the UMD property due to the Mihlin multiplier theorem \cite[Theorem 6.2.7]{Graf} and its generalization to the UMD space case \cite[Theorem 5.5.10]{HNVW1} (in fact the $C^{\infty}$-condition can be weakened a lot: due to Theorem 3.1 and Corollary 4.1 in \cite{Hyt10} it is sufficient to have $m|_{\mathbb R^d\setminus \{0\}} \in C^{\lfloor \frac d2\rfloor + 1}$). Nevertheless, by the latter remark it is not true that $T_m$ is bounded on $L^p(\mathbb R^d)$ for a general bounded even homogeneous $m$. Therefore it is not true that $T_m$ is bounded on $L^p(\mathbb R^d)$ with a general continuous even homogeneous $m$. Indeed, assume that $\|T_{m}\|_{\mathcal L(L^p(\mathbb R^d))}<\infty$ for any continuous even homogeneous $m$. Then by the closed graph theorem, $m\mapsto T_m$ is a bounded operator, and $\|T_{m}\|_{\mathcal L(L^p(\mathbb R^d))}\leq C\|m\|_{\infty}$ 
 for some $C>0$. But hence the same estimate holds for a general bounded even homogeneous $m$: if $(m_n)_{n\geq 0}$ is a sequence of continuous even homogeneous functions such that $m_n \to m$ pointwise as $n\to \infty$ and $\|m_n\|_{\infty}\leq \|m\|_{\infty}$, then for any Schwartz function $f:\mathbb R^d\to \mathbb C$
 \[
  \|T_m f\|_{L^p(\mathbb R^d)}\leq \liminf_{n\to \infty} \|T_{m_n} f\|_{L^p(\mathbb R^d)} \leq C\|f\|_{L^p(\mathbb R^d)},
 \]
 which contradicts with Remark \ref{rem:conterexevenhomandunbdd}.
\end{remark}

\bibliographystyle{plain}

\begin{thebibliography}{10}

\bibitem{AH86}
R.J. Aumann and S.~Hart.
\newblock Bi-convexity and bi-martingales.
\newblock {\em Israel J. Math.}, 54(2):159--180, 1986.

\bibitem{BO12}
R.~Ba\~nuelos and A.~Os{\c{e}}kowski.
\newblock Martingales and sharp bounds for {F}ourier multipliers.
\newblock {\em Ann. Acad. Sci. Fenn. Math.}, 37(1):251--263, 2012.

\bibitem{BBB}
R.~Ba{\~n}uelos, A.~Bielaszewski, and K.~Bogdan.
\newblock Fourier multipliers for non-symmetric {L}\'evy processes.
\newblock In {\em Marcinkiewicz centenary volume}, volume~95 of {\em Banach
  Center Publ.}, pages 9--25. Polish Acad. Sci. Inst. Math., Warsaw, 2011.

\bibitem{BB}
R.~Ba{\~n}uelos and K.~Bogdan.
\newblock L\'evy processes and {F}ourier multipliers.
\newblock {\em J. Funct. Anal.}, 250(1):197--213, 2007.

\bibitem{BP12}
V.~Barbu and T.~Precupanu.
\newblock {\em Convexity and optimization in {B}anach spaces}.
\newblock Springer Monographs in Mathematics. Springer, Dordrecht, fourth
  edition, 2012.

\bibitem{Bour83}
J.~Bourgain.
\newblock Some remarks on {B}anach spaces in which martingale difference
  sequences are unconditional.
\newblock {\em Ark. Mat.}, 21(2):163--168, 1983.

\bibitem{Bour}
J.~Bourgain.
\newblock Vector-valued singular integrals and the {$H^1$}-{BMO} duality.
\newblock In {\em Probability theory and harmonic analysis ({C}leveland,
  {O}hio, 1983)}, volume~98 of {\em Monogr. Textbooks Pure Appl. Math.}, pages
  1--19. Dekker, New York, 1986.

\bibitem{Burk84}
D.L. Burkholder.
\newblock Boundary value problems and sharp inequalities for martingale
  transforms.
\newblock {\em Ann. Probab.}, 12(3):647--702, 1984.

\bibitem{Burk89}
D.L. Burkholder.
\newblock Differential subordination of harmonic functions and martingales.
\newblock In {\em Harmonic analysis and partial differential equations ({E}l
  {E}scorial, 1987)}, volume 1384 of {\em Lecture Notes in Math.}, pages 1--23.
  Springer, Berlin, 1989.

\bibitem{Burk01}
D.L. Burkholder.
\newblock Martingales and singular integrals in {B}anach spaces.
\newblock In {\em Handbook of the geometry of {B}anach spaces, {V}ol. {I}},
  pages 233--269. North-Holland, Amsterdam, 2001.

\bibitem{CS95}
A.~Carbery and A.~Seeger.
\newblock Homogeneous {F}ourier multipliers of {M}arcinkiewicz type.
\newblock {\em Ark. Mat.}, 33(1):45--80, 1995.

\bibitem{Choi92}
K.P. Choi.
\newblock A sharp inequality for martingale transforms and the unconditional
  basis constant of a monotone basis in {$L^p(0,1)$}.
\newblock {\em Trans. Amer. Math. Soc.}, 330(2):509--529, 1992.

\bibitem{Gar85}
D.J.H. Garling.
\newblock Brownian motion and {UMD}-spaces.
\newblock In {\em Probability and {B}anach spaces ({Z}aragoza, 1985)}, volume
  1221 of {\em Lecture Notes in Math.}, pages 36--49. Springer, Berlin, 1986.

\bibitem{GM-SS}
S.~Geiss, S.~Montgomery-Smith, and E.~Saksman.
\newblock On singular integral and martingale transforms.
\newblock {\em Trans. Amer. Math. Soc.}, 362(2):553--575, 2010.

\bibitem{Graf}
L.~Grafakos.
\newblock {\em Classical {F}ourier analysis}, volume 249 of {\em Graduate Texts
  in Mathematics}.
\newblock Springer, New York, third edition, 2014.

\bibitem{Hor60}
L.~H\"ormander.
\newblock Estimates for translation invariant operators in {$L^{p}$}\ spaces.
\newblock {\em Acta Math.}, 104:93--140, 1960.

\bibitem{HNVW1}
T.~Hyt\"{o}nen, J.M.A.M. van Neerven, M.C. Veraar, and L.~Weis.
\newblock {\em Analysis in {B}anach spaces. {V}ol. {I}. {M}artingales and
  {L}ittlewood-{P}aley theory}, volume~63 of {\em Ergebnisse der Mathematik und
  ihrer Grenzgebiete. 3.}
\newblock Springer, 2016.

\bibitem{Hyt10}
T.P. Hyt\"onen.
\newblock New thoughts on the vector-valued {M}ihlin-{H}\"ormander multiplier
  theorem.
\newblock {\em Proc. Amer. Math. Soc.}, 138(7):2553--2560, 2010.

\bibitem{JT}
M.~Jouak and L.~Thibault.
\newblock Directional derivatives and almost everywhere differentiability of
  biconvex and concave-convex operators.
\newblock {\em Math. Scand.}, 57(1):215--224, 1985.

\bibitem{Kal}
O.~Kallenberg.
\newblock {\em Foundations of modern probability}.
\newblock Probability and its Applications (New York). Springer-Verlag, New
  York, second edition, 2002.

\bibitem{deL65}
K.~de Leeuw.
\newblock On {$L_{p}$} multipliers.
\newblock {\em Ann. of Math. (2)}, 81:364--379, 1965.

\bibitem{McC84}
T.R. McConnell.
\newblock On {F}ourier multiplier transformations of {B}anach-valued functions.
\newblock {\em Trans. Amer. Math. Soc.}, 285(2):739--757, 1984.

\bibitem{Pis16}
G.~Pisier.
\newblock {\em Martingales in Banach spaces}, volume 155.
\newblock Cambridge University Press, 2016.

\bibitem{Rubio86}
J.L. Rubio~de Francia.
\newblock Martingale and integral transforms of {B}anach space valued
  functions.
\newblock In {\em Probability and {B}anach spaces ({Z}aragoza, 1985)}, volume
  1221 of {\em Lecture Notes in Math.}, pages 195--222. Springer, Berlin, 1986.

\bibitem{Thib84}
L.~Thibault.
\newblock Continuity of measurable convex and biconvex operators.
\newblock {\em Proc. Amer. Math. Soc.}, 90(2):281--284, 1984.

\bibitem{Y17FourUMD}
I.S. Yaroslavtsev.
\newblock Fourier multipliers and weak differential subordination of
  martingales in {UMD} {B}anach spaces.
\newblock {\em arXiv:1703.07817. To appear in Studia Math.}, 2017.

\bibitem{Y17MartDec}
I.S. Yaroslavtsev.
\newblock Martingale decompositions and weak differential subordination in
  {U}{M}{D} {B}anach spaces.
\newblock {\em arXiv:1706.01731}, 2017.

\end{thebibliography}

\def\cprime{$'$} \def\polhk#1{\setbox0=\hbox{#1}{\ooalign{\hidewidth
  \lower1.5ex\hbox{`}\hidewidth\crcr\unhbox0}}}
  \def\polhk#1{\setbox0=\hbox{#1}{\ooalign{\hidewidth
  \lower1.5ex\hbox{`}\hidewidth\crcr\unhbox0}}} \def\cprime{$'$}

\end{document}